\documentclass[10pt]{article}
\usepackage[a4paper]{geometry}
\usepackage[utf8]{inputenc}
\usepackage[T1]{fontenc}
\usepackage[english]{babel}
\usepackage[babel=true,kerning=true]{microtype}
\usepackage{lmodern}

\usepackage{pifont}
\rmfamily
\DeclareFontShape{T1}{lmr}{b}{sc}{<->ssub*cmr/bx/sc}{}
\DeclareFontShape{T1}{lmr}{bx}{sc}{<->ssub*cmr/bx/sc}{}

\usepackage{tikz}
\usepackage{wrapfig} 
\usepackage[all]{xy}

\usetikzlibrary{shapes,patterns,plotmarks}
\usetikzlibrary{arrows}
\usepackage{amsfonts,amssymb}
\usepackage{amsmath,amsthm,amscd}

\usepackage{bbm}

\usepackage{array}
\usepackage{epsfig,graphics,graphicx,xcolor}
\usepackage{stmaryrd}

\usepackage[babel=true]{csquotes}
\usepackage{enumitem}
\usepackage{mathtools}
\usepackage[numbers]{natbib}

\usepackage{color}

\definecolor{lightred}{rgb}{1, 0.5, 0.5}
\definecolor{lightgrey}{rgb}{0.6, 1, 0.6}

\DeclarePairedDelimiter\abs{\lvert}{\rvert}%

\DeclareFontFamily{U}{mathx}{\hyphenchar\font45}
\DeclareFontShape{U}{mathx}{m}{n}{
      <5> <6> <7> <8> <9> <10>
      <10.95> <12> <14.4> <17.28> <20.74> <24.88>
      mathx10
      }{}
\DeclareSymbolFont{mathx}{U}{mathx}{m}{n}
\DeclareMathSymbol{\bigtimes}{1}{mathx}{"91}

    \makeatletter
    \newcounter {subsubsubsection}[subsubsection]
    \renewcommand\thesubsubsubsection{\thesubsubsection .\@arabic\c@subsubsubsection}
    \newcommand\subsubsubsection{\@startsection{subsubsubsection}{4}{\z@}%
                                         {-3.25ex\@plus -1ex \@minus -.2ex}%
                                         {1.5ex \@plus .2ex}%
                                         {\normalfont\normalsize\itshape}}
    \newcommand*\l@subsubsubsection{\@dottedtocline{4}{10.0em}{4.1em}}
    \newcommand*{\subsubsubsectionmark}[1]{}
    \makeatother

  \theoremstyle{plain}
\newtheorem{theorem}{Theorem}[section]
\newtheorem{proposition}[theorem]{Proposition}
\newtheorem{corollary}[theorem]{Corollary}
\newtheorem{conjecture}[theorem]{Conjecture} 
\newtheorem{lemma}[theorem]{Lemma} 
\newtheorem{conventiondef}[theorem]{Convention-Definition} 
 
  \theoremstyle{remark}
\newtheorem{remark}[theorem]{Remark}
  \theoremstyle{definition}
\newtheorem{definition}[theorem]{Definition}

\newtheorem{question}[theorem]{Question}

\newcommand{\TI}
{T\! \circ \!I}
\newcommand{\ie}{\textit{i.e. }}

\newcommand{\ZZ}{\mathbb{Z}}
\newcommand{\RR}{\mathbb{R}}
\renewcommand{\SS}{\mathbb{S}}

\newcommand{\QQ}{\mathbb{Q}}

\newcommand{\nab}{\raisebox{-.65ex}{\ensuremath{\nabla}}}

\renewcommand{\a}{\alpha_3^1}

\newcommand{\<}{\left<\!\left<}
\renewcommand{\>}{\right>\!\right>}

\title{Finite-type $1$-cocycles of knots given by Polyak-Viro formulas}
\author{Arnaud Mortier}
\date{\today}

\begin{document}

\maketitle

\begin{abstract}
\footnotesize
We present a new method to produce simple formulas for $1$-cocycles of knots over the integers, inspired by Polyak-Viro's formulas for finite-type knot invariants. We conjecture that these $1$-cocycles represent finite-type cohomology classes in the sense of Vassiliev. An example of degree $3$ is studied, and shown to coincide over $\ZZ_2$ with the Teiblum-Turchin cocycle $v_3^1$.

\end{abstract}

\tableofcontents

\section*{Introduction}

The study of the topology of the space of knots was initiated in $1990$ by V.A.Vassiliev \citep{Vassiliev1990}, who defined finite-type cohomology classes by applying ideas from the finite-dimensional affine theory of plane arrangements to the infinite-dimensional theory of long knots in $3$-space. Here, \textit{of finite type} means of finite complexity, in some sense, hence hopefully computable.
Independently, outstanding general results on the topology of knot spaces have been obtained by A.Hatcher \citep{HatcherTopologicalmoduli}, leading to the idea that higher dimensional invariants of knots -- in particular, $1$-cocycles -- should capture information about the geometry of a knot (see \citep{FiedlerQuantum1cocycles}).

The zeroth level of Vassiliev's theory, known as finite-type knot invariants, has been extensively studied in the subsequent years \citep{BirmanLin, Kontsevich, Goussarov, BarNatanVKI}.  However at the first level -- that of $1$-cocycles, only one example, in degree $3$, has been proved to exist by D.Teiblum and V.Turchin, and then actually described by V.A.Vassiliev with a formula over $\ZZ_2$ \citep{Vassiliev, Turchin}. Since then, no progress has been made, probably because of the technicity of Vassiliev's construction and the apparent difficulty of turning it into a systematic method: indeed, it involves singularity theory with differential geometric conditions. The theory seems to be stuck at this one question: is there a way to evaluate finite-type $1$-cocycles combinatorially, \ie without involving geometry?

Our main result is that the answer is yes. In \citep{MortierPolyakEquations, VKTG}, it was shown that Polyak-Viro's arrow diagram formulas \citep{PolyakViro} (the most compact combinatorial formulas to describe finite-type invariants) are the kernel of an explicit linear map with values in a space of degenerate arrow diagrams. In Section~\ref{sec:top} we show that this map has a natural interpretation as a coboundary map in a simplicial model of the space of knots, adapted to the particular feature of Polyak-Viro's formulas that is \textit{counting subdiagrams}. As a result, the $1$-codimensional objects in this simplicial model are natural candidates to describe $1$-cocycles in the space of knots.

Concretely, just like Polyak-Viro's invariants are defined via the finite set of crossings of a knot projection, the ways they are arranged (the Gauss diagram), and its power set (the subdiagrams, endowed with weights that are products of writhe numbers), we consider the finite set of \textit{germs} of a loop in the space of knots endowed with a projection -- that is, its local behaviour near each Reidemeister move, together with their \textit{subgerms}, also weighted with the product of their writhe numbers. Cocycles defined via knot projections should vanish on the meridians of the $2$-codimensional strata defined by the \textit{higher order Reidemeister theorem} \citep{Fiedler1parameter, Fiedler1parameterPolynomials} -- these strata also appear as a set of elementary moves in the study of surfaces embedded in $4$-space, known  as \textit{Roseman moves}, in their \enquote{movie} version \citep{CarterKamadaSaito, CarterSaito}. In Section~\ref{sec:method}, we explain how to write down the system of equations derived from those strata, and how to reduce it so as to make it computable. 

In the last section, we compute  that system of equations in degree $3$ and study the properties of one of its solutions which we call $\alpha_3^1$. It is proved in that the evaluation of $\alpha_3^1$ on the rotation of a long knot $K$ around its axis is equal to $-v_2(K)$. This was conjectured to hold (up to sign) for the Teiblum-Turchin cocycle in \cite{Turchin}. Lastly, we prove that the reductions mod $2$ of $\alpha_3^1$ and the Teiblum-Turchin cocycle are equal by showing directly that $\alpha_3^1$ mod $2$ is of finite type.

On the basis of these facts, we make the following conjecture.

\paragraph*{Conjecture 1.} Every $1$-cocycle of knots defined by an arrow germ formula of degree $n$ represents a finite-type cohomology class of degree no greater than $n$.\\
 
As an immediate consequence of this conjecture, one would obtain:
\paragraph*{Conjecture 2.} The arrow germ formula $\alpha_3^1$ is a realization of Teiblum-Turchin's cocycle over $\ZZ$.

\section*{Acknowledgements} 

I wish to express my full gratitude to Seiichi Kamada for his invitation to the Osaka City University Advanced Mathematical Insitute where this work was done, and his help through the Japanese administration. I am also grateful to Victoria Lebed for her inestimable support.

I thank Thomas Fiedler, Michael Polyak and Victor Turchin for fruitful discussions.

\section{Knot invariants from a simplicial viewpoint}
\label{sec:top}

\subsection{General principle}\label{subsec:principle}

One of the most common ways for a knot theorist to define a knot invariant consists of two steps: first, start from a knot diagram and construct something from it, say an element of an abelian group $A$; then, prove that the result does not change when one performs Reidemeister moves. There is an obvious interpretation of this process in terms of cellular cohomology with coefficients in $A$. Indeed, one may think of the set of knot diagrams isotopic to a given one as a cell, bounded by (finitely many, two-sided) codimension $1$ cells that are Reidemeister moves: associating something with every knot diagram amounts then to the choice of a $0$-cochain, and by applying the Stokes formula one sees that this cochain defines a knot invariant if and only if it is a cocycle.

Of course, this cellular complex is huge and there is no hope to compute its first coboundary map in general. However, assume that one has constructed a simple family of maps $\mathcal{F}$ from the set of knot diagrams to $A$, and wonders \emph{which linear combinations of these maps give knot invariants?} The answer can be given by an adequate model for the coboundary map from the previous discussion. Namely:

\begin{enumerate}
\item Construct a set $\mathcal{F}^1$ of objects that are likely to represent how the elements of $\mathcal{F}$ behave under Reidemeister moves.
\item For each element $f^1\in\mathcal{F}^1$ and each Reidemeister move $D\rightsquigarrow D^\prime$ associate an element of $A$, denoted by $\left<f^1, D\rightsquigarrow D^\prime\right>$.

\item Construct a linear map $d:\ZZ \mathcal{F}\rightarrow \ZZ \mathcal{F}^1$ that satisfies the formula:
$$f(D^\prime)-f(D)=\left<d(f), D\rightsquigarrow D^\prime\right> .$$
\end{enumerate}

Here $\ZZ \mathcal{F}$ is the $\ZZ$-module freely generated by the elements of $\mathcal{F}$. In other words, one has to build every ingredient in a Stokes formula, and prove that the formula is actually satisfied. There is a two-sided benefit from such a construction. First, the initial goal was achieved since computing the invariants coming from the family of maps $\mathcal{F}$ has been reduced to computing the kernel of an explicit matrix. Second, the elements of $\mathcal{F}^1$ are now very good candidates to produce $1$-cocycles of knots.

By $1$-cocycle of knots, we mean a $1$-cocycle in the space of all smooth embeddings $\SS^1 \hookrightarrow \RR^3$. The simplicial viewpoint can be extended one step further, thanks to the higher dimensional Reidemeister theorem that describes codimension $2$ strata corresponding to the \enquote{simplest} degeneracies of Reidemeister moves (see 
\cite{Fiedler1parameter, Fiedler1parameterPolynomials}). These strata will be described in Section~\ref{sec:strata}. We now give our main example of a family of maps $\mathcal{F}$ that fits nicely into such a cohomological framework.

\subsection{The example of arrow diagram invariants}

In \citep{MortierPolyakEquations} and \citep{VKTG} it has been shown that Goussarov-Polyak-Viro's combinatorial formulas for Vassiliev invariants (\cite{GPV}) are the kernel of an explicit linear map. We show here that this map is a part of a $0$-coboundary map in the spirit of the previous section. The framework is that of \textit{long knots} and \textit{based Gauss diagrams}, because they have easier combinatorics, and because it was the original settings for which V.A.Vassiliev defined his finite-type cohomology \citep{Vassiliev1990}. However, every statement can be adapted to other kinds of $1$-dimensional knotted objects such as the virtual knots on a group introduced in \citep{VKTG}, or tangle diagrams in the disc -- paying attention to symmetries whenever they may appear (see \citep{Ostlund}, Sections~$2.2$ and $2.4$, and \citep{VKTG}, Section~$4.1.2$).

\subsubsection{Notations and basic notions}\label{sec:notations}

\textbf{Convention.
}When several incomplete diagrams are represented side by side in a picture, or in one and the same equation, it is to be understood that
\begin{enumerate}
\item Every unseen part -- including missing decorations, such as local orientations -- is the same for all diagrams. 
\item The picture is valid no matter what are those unseen parts, unless otherwise specified in the caption. However, the \textit{visible} parts cannot contain additional arrow ends or a missing point at infinity.
\end{enumerate}

\subsubsubsection{Notations for the $0$-codimensional data}
A \textit{(based) arrow diagram} is an oriented line together with a finite number $n$ of abstract oriented chords attached to it at $2n$ distinct points. We shall represent such a diagram as based on a circle with a distinguished point \enquote{at infinity}. A \textit{(based) Gauss diagram} is an arrow diagram in which every arrow has been decorated with a sign, $+$ or $-$. Both types of diagrams are regarded up to positive homeomorphisms of the real line.

\textbf{Fact.} (see \cite{KauffmanVKT99}) Gauss diagrams are in $1$-$1$ correspondence with (virtual) long knot diagrams in the plane $\RR^2$ up to diagram isotopy (and detour moves).
 
This means that a Gauss diagram represents exactly what we want to think of as a cell of maximal dimension. Consequently, we introduce the $\QQ$-space $\mathfrak{G}$ freely generated by all Gauss diagrams and think of it as our space of $0$-chains. It is graded by the number of arrows.

Gauss diagrams enjoy a set of \textit{R-moves}, which are equivalent to the usual Reidemeister moves for knot diagrams. Their skeleton is depicted in Fig.\ref{Rmoves}; the signs and orientations have to follow some rules explained thereafter.

\begin{figure}[h!]
\centering 
\psfig{file=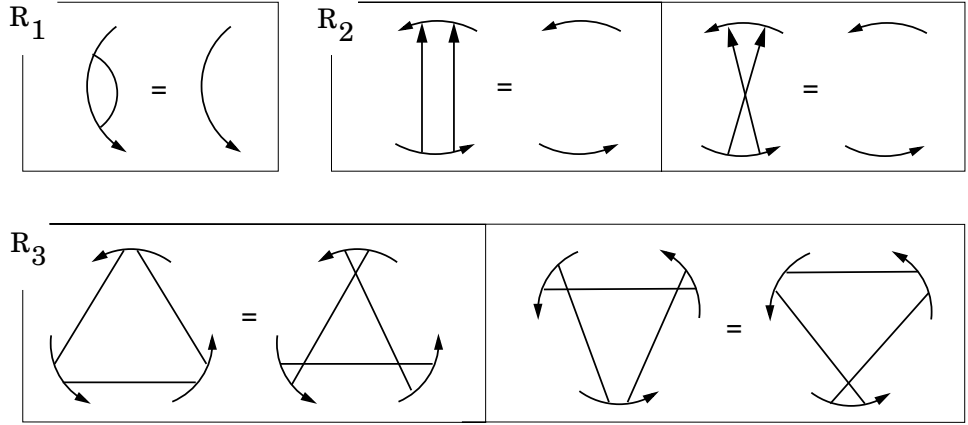,scale=0.7}
\caption{$\textrm{R}$-moves for Gauss diagrams (see below the rules for the decorations)}\label{Rmoves}
\end{figure}

\subsubsection*{R$_1$-moves}

An \textit{R$_1$-move} is the birth or death of an isolated arrow, as shown in Fig.\ref{Rmoves} (top-left). There is no restriction on the direction or the sign of the arrow.

\subsubsection*{R$_2$-moves}

An \textit{R$_2$-move} is the birth or death of a pair of arrows with different signs, whose heads are consecutive as well as their tails (Fig.\ref{Rmoves}, top-right).

\subsubsection*{R$_3$-moves}

\begin{definition}[The sign $\varepsilon$ and the co-orientation of R$_3$ moves]\label{def_epsilon}
In a classical Gauss diagram of degree $n$, the complementary of the arrows is made of $2n$ oriented components. These are called the \textit{edges} of the diagram. In a diagram with no arrow, we still call the whole circle an edge.

Let $e$ be an edge in a Gauss diagram, between two consecutive arrow ends that do not belong to the same arrow.
Put
$$\eta (e)=\left\lbrace \begin{array}{l}
+1 \text{ if the arrows that bound $e$ cross each other} \\
-1 \text{ otherwise}
\end{array}\right. ,$$
and let $\uparrow\!\!(e)$ be the number of arrowheads at the boundary of $e$.
Then define
$$\varepsilon (e)= \eta(e)\cdot(-1)^{\uparrow(e)}.$$
Finally, define $\mathrm{w}(e)$ as the product of the signs of the two arrows at the boundary of $e$.

\end{definition}

An \textit{R$_3$-move} is the simultaneous switch of the endpoints of three arrows as shown on Fig.\ref{Rmoves} (bottom), with the following conditions:
\begin{enumerate}
\item The value of $\mathrm{w}(e)\varepsilon(e)$ should be the same for all three visible edges $e$. This ensures that the piece of diagram containing the three arrows can be represented in a knot-diagrammatic way without making use of virtual crossings.
\item The values of  $\uparrow\!\!(e)$ should be pairwise different. This ensures that one of the arcs in the knot diagram version actually \enquote{goes over} the others.
\end{enumerate}

\begin{remark}\label{rem:coor}
From our cellular viewpoint, the sign $\mathrm{w}(e)\varepsilon(e)$ gives a natural co-orientation of the $1$-codimensional strata corresponding to R$_3$ moves. This \enquote{good} co-orientation is part of the reason why the $1$-cocycle formulas we are going to find later are so simple.
\end{remark}

\begin{proof}[Proof that our criterion for R$_3$-moves is correct]
Condition $2$ is obviously necessary. Now assume that it is satisfied. 
It is enough to check only one case where Condition $1$ is satisfied, and one case where it fails. Indeed, all possible cases are linked by an abstract connected graph (the \textit{cube} of R$_3$-moves, see Fig.\ref{cube}), and the status of Condition $1$ is preserved by walking along this graph. Namely, two adjacent triangles from Fig.\ref{cube} differ at the level of Gauss diagrams by switching two consecutive arrow ends, and changing the sign of the third arrow.
\end{proof}

\begin{figure}[h!]
\centering 
\psfig{file=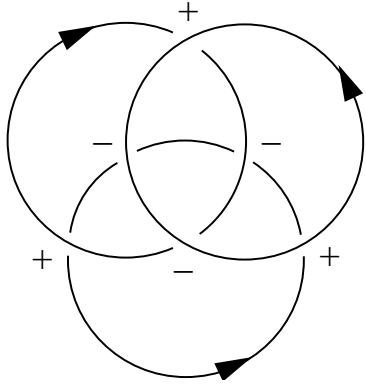,scale=0.78}
\caption{The cube of all eight possible local R-III situations and \enquote{edges} between them -- the diagram lies on $\SS^2$, and the \enquote{unbounded} region is the eighth triangle.}\label{cube}
\end{figure}

Any arrow diagram $A$ defines a linear form on $\mathfrak{G}$ by the Polyak-Viro formula \citep{PolyakViro}:
$$\<A,G\>=\left<S(A),I(G)\right>,$$
where $S(A)=\sum \operatorname{sign}(\sigma)A^\sigma$ is the alternate sum of all ($2^{\operatorname{deg}(A)}$) completions of $A$ into a Gauss diagram, $I(G)$ is the formal sum of all ($2^{\operatorname{deg}(G)}$) subdiagrams of $G$, and $\left<,\right>$ is the orthonormal scalar product with respect to the basis of Gauss diagrams. Roughly speaking, $\<A,G\>$ counts the number of times the configuration $A$ happens in $G$, with weights given by the product of the signs of the arrows involved.

The set of linear forms induced by arrow diagrams is our family of maps $\mathcal{F}$ in these settings (see Section~\ref{subsec:principle}).
The $\QQ$-space freely generated by this set is denoted by $\mathfrak{A}$ and regarded as a space of $0$-cochains. A linear combination $\mathcal{A}\in
\mathfrak{A}$ such that the map $\<\mathcal{A},\cdot\>$ is invariant under R-moves is called a \textit{(virtual) arrow diagram formula}. 

\subsubsubsection{Notations for the $1$-codimensional data}

\begin{definition}[germs, subgerms, partial germs]
An \textit{i-germ}, for $i=1,2,3$, is an ordered couple $(G_0, G_1)$ of Gauss diagrams that differ by an R$_i$-move. To prevent possible ambiguity, the edges involved in the R-moves are distinguished -- in the pictures, we will represent this by little dots (that shall not be confused with the point at infinity, always flanked by an $\infty$ sign). An arrow adjacent to at least one distinguished edge is called \textit{distinguished} as well.

Due to the nature of arrow diagram formulas (counting \textit{subdiagrams}), we need to introduce the notions of \textit{subgerms} and \textit{partial germs}. Note that there is a natural correspondence between the arrows of the two diagrams in an $i$-germ (excluding those involved in the R-move in cases $i=1,2$).

A \textit{partial $3$-germ} (or simply \textit{partial germ}) is the result of removing one distinguished arrow from a $3$-germ. The only intact distinguished edge is still called distinguished (as well as the adjacent arrows).

A \textit{subgerm} of a germ $(G_0,G_1)$ is the result of removing a (maybe empty) set of arrows from $G_0$, together with their match in $G_1$ -- in particular, no distinguished arrow can be removed in cases $i=1,2$. Also, in case $i=3$, at most one distinguished arrow may be removed. Hence a subgerm is always either a germ itself, or a partial germ.

The notions of \textit{arrow $i$-germs}, \textit{arrow subgerms} and  \textit{partial arrow germs} are defined similarly, with no signs decorating the arrows. 

The $\QQ$-spaces generated by $i$-germs ($i=1,2,3$) and partial germs, modulo all the relations $(x,y)+(y,x)=0$ in case $i=3$, are respectively denoted by $\mathfrak{G}_I$, $\mathfrak{G}_{I\!I}$, $\mathfrak{G}_\Delta$ and $\mathfrak{G}_\Lambda$. The corresponding arrow diagram spaces are respectively denoted by $\mathfrak{A}^I$, $\mathfrak{A}^{I\!I}$, $\mathfrak{A}^\Delta$ and $\mathfrak{A}^\Lambda$. 

The space $\mathfrak{A}^\Lambda$ is meant to be modded out by the \textit{triangle relations}, that can be of two types. One of them is shown on Fig.\ref{trirel}, the other is obtained by reversing all the arrows in the picture. The quotient is denoted by $\mathfrak{A}^\Lambda/\nab$. 

Finally, set
$$\mathfrak{G}_1=
\mathfrak{G}_I\oplus
\mathfrak{G}_{I\!I}\oplus
\mathfrak{G}_\Delta,\hspace*{0.5cm}\text{ and }\hspace*{0.5cm}
\mathfrak{A}^1=
\mathfrak{A}^I\oplus
\mathfrak{A}^{I\!I}\oplus
\mathfrak{A}^\Delta\oplus
\mathfrak{A}^\Lambda    /\nab.$$
These are respectively our spaces of $1$-chains and $1$-cochains.
\end{definition}

\begin{remark}
The triangle relations have appeared in several places in topology since M.Polyak found them at the beginning of his work on arrow diagrams enhanced by $3$-vertices \citep{P1} -- see his review given in \citep{PolyakTalk}. They also appear tautologically in V.A.Vassiliev's language in \citep{Vassiliev}.  Here, their meaning is the following: we have seen that in a cellular model, being a knot invariant should be equivalent to having zero coboundary, and obviously, zero and zero \textit{modulo something} are not equivalent. However, partial arrow germs are not strictly a part of the cellular model that we discussed earlier: they are an artifact due to the particular settings of arrow diagram invariants that regard subdiagrams rather than simply diagrams. To recover an adapted cellular model, one should dilate the R$_3$ strata so that the three arrow-end switches happen separately. The oriented boundary of a dilated stratum is then exactly a triangle relator (this is explained in detail in \citep{PVCasson}).
\end{remark}

\begin{figure}[h!]
\centering 
\psfig{file=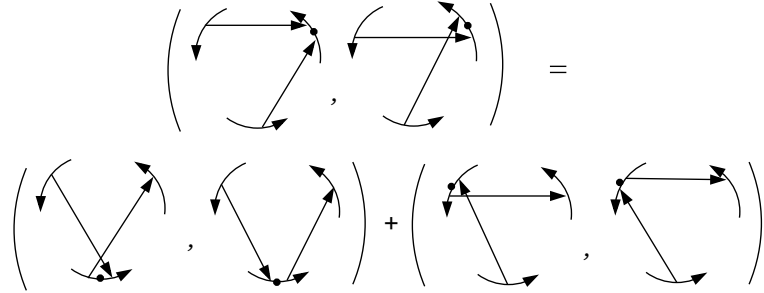,scale=1}
\caption{One triangle relation for partial arrow  germs -- the other is obtained by reversing the orientation of every arrow}\label{trirel} 
\end{figure}

\subsubsection{Stokes formula}

\subsubsubsection{The  boundary and coboundary maps}
The boundary map $\partial$ is defined on the generators by
$$\begin{array}{cccc}
\partial : & \mathfrak{G}_1 & \rightarrow & \mathfrak{G}\\
 & (G_0, G_1) & \mapsto & G_1 - G_0
\end{array}.$$
It obviously maps to zero the relators $(x,y)-(y,x)$.

\begin{definition}\label{def:coboundary}
Let $A$ be an arrow diagram. Consider the formal sum of all arrow germs and partial arrow germs $(A^0, A^1)$ such that $A^1$ (after forgetting that some edges are distinguished) is equal to $A$, and \textit{such that $A^1$ has no less arrows than $A^0$}.  The class of this formal sum modulo the triangle relations is denoted by $dA$.
This defines componentwise a linear map
$$d:\mathfrak{A}\rightarrow \mathfrak{A}^1,$$ that splits into $d^I\oplus d^{I\!I}\oplus d^\Delta\oplus d^\Lambda $.
\end{definition}

\begin{remark}
The reason why $d^I$ and $d^{I\!I}$ do not count the germs $(A^0, A^1)$ with $\operatorname{deg}(A^1)<\operatorname{deg}(A^0)$ is the following. By essence, when an arrow diagram looks at a Gauss diagram, what it sees is all of the subdiagrams of $G$. Now in the boundary of an $i$-germ ($i=1,2$), the Gauss diagram with less arrows happens \textit{twice} as a subdiagram, with opposite coefficients, whence arrow diagrams do not see it at all.
\end{remark} \subsubsubsection{The pairing between $\mathfrak{G}_1$ and $\mathfrak{A}^1$ }
We mimic the definition of arrow diagrams as $0$-cochains. The space $\mathfrak{G}_1\oplus \mathfrak{G}_\Lambda$ is endowed with  the orthonormal scalar product with respect to the basis of germs, denoted by $\left<  ,\right>$.
Again there is a map $S$ defined componentwise by the formula ($*\in \left\lbrace I, II, \Delta, \Lambda \right\rbrace $):

$$
\begin{array}{cccc}
S: &  \mathfrak{A}^* & \rightarrow & \mathfrak{G}_*\\
 & \alpha & \mapsto & \sum_{\sigma\in \left\lbrace \pm 1\right\rbrace^n}\operatorname{sign}(\sigma) \alpha^\sigma 
\end{array},$$ 
where $\alpha^\sigma$ is the enhancement of $\alpha$ into a germ by sign decorations, and $\operatorname{sign}(\sigma)$ is the product of these signs. This map admits an adjoint for $\left<, \right>$, denoted by $T$, which forgets the signs of a germ and remembers only their product as a coefficient (see \citep{VKTG}, Section~$4.1.2$).

Finally, we introduce the linear map
$$I:\mathfrak{G}_1\rightarrow  \mathfrak{G}_1\oplus \mathfrak{G}_\Lambda$$
that sends a germ to the formal sum of its subgerms.

\begin{definition}
For $\alpha\in \mathfrak{A}^I\oplus
\mathfrak{A}^{I\!I}\oplus
\mathfrak{A}^\Delta\oplus
\mathfrak{A}^\Lambda $ and $\gamma\in \mathfrak{G}_1$, set
$$\<  \alpha,\gamma\> =\left<  S(\alpha),
I(\gamma)\right>=\left<  \alpha,
\TI(\gamma)\right>.$$

\end{definition}

\begin{lemma}\label{cruciallemma1}
The value of the bracket $\<  \alpha,\gamma\>$ only depends on the class of $\alpha$ in $\mathfrak{A}^1$.
\end{lemma}

This is a consequence of a deeper fact: let $\alpha=(A^0,A^1)$ denote an arrow germ. Assigning sign decorations to the arrows as usual with a map $\sigma:\left\lbrace 1,\ldots,\operatorname{deg}(A^0)\right\rbrace\rightarrow\left\lbrace \pm1\right\rbrace $, consistently for the two diagrams, one obtains a \textit{formal $3$-germ} $\alpha^\sigma$, which may or may not actually be a $3$-germ (\ie come from an $\mathrm{R}_3$ move). 
It happens that the triangle relations and our pairing $\<,\>$ detect that:

\begin{lemma}With the above notations, $\alpha^\sigma$ is an actual $3$-germ if and only if
for every triangle relator $\nabla$ (Fig.\ref{trirel}), $$
\<  \nabla,\alpha^\sigma\>=0.
$$
\end{lemma}

\begin{proof} Note that the definition of the pairing $\< , \>$ obviously extends to formal $3$-germs, so that the lemma makes sense:
$$
\<  \nabla,\alpha^\sigma\>= \left< \nabla,\TI(\alpha^\sigma)\right> .
$$ 
We look only at partial $3$-germs in the sum $I(\alpha^\sigma)$ since the other terms clearly do not take part in the computation. These can be grouped in triples, in which any two partial germs differ only by the arrow that they miss from an R$_3$ triangle (see an example on Fig.~\ref{trirelsigned}).

\begin{figure}[h!]
\centering 
\psfig{file=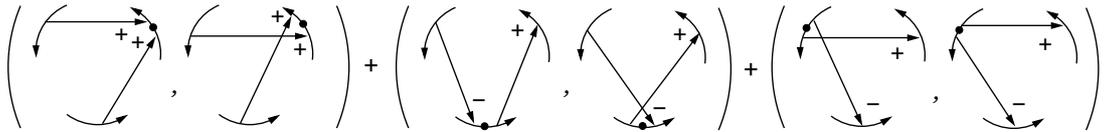,scale=1}
\caption{One possible triple in the sum $I(\alpha^\sigma)$}\label{trirelsigned} 
\end{figure}

Pick such a triple, say $\lambda$, and consider the sum $\left<  \nabla,T(\lambda)\right>$. It has $3*3=9$ terms, and it is easy to see that either they are all zero, or two of them are non zero. In the latter case, we make the following observations.

$1.$ The two partial arrow germs non trivially involved appear in $\nabla$ with opposite coefficients - indeed, the two at the bottom of Figure~\ref{trirel} cannot both appear in $\lambda$. 

$2.$ The product of the signs of the non-distinguished arrows is the same for all germs on Fig.~\ref{trirelsigned}.

$3.$ For each partial germ $(G_0, G_1)$ from Fig.~\ref{trirel}, one has $\varepsilon (G_1)=+1$. 

Thus $\left<  \nabla,T(\lambda)\right>$ is zero if and only if computing $\varepsilon$ times the product of signs of the visible arrows gives the same result for both relevant diagrams in $\lambda$, which is precisely a consequence of the fact that $\lambda$ comes from an $\mathrm{R}_3$ move (Section~\ref{sec:notations}). This proves the \enquote{only if} part.\\

For the converse, pick a $\nabla$ relation of degree two (\ie without unseen arrows in Fig.~\ref{trirel}). Then the only relevant terms in the sum $I(\alpha^\sigma)$ are those obtained by forgetting all arrows but two from the R$_3$ triple. There are only three such partial germs. The computation is the same as before, and the assumption that the bracket is zero allows one to conclude, using again the description of $\mathrm{R}_3$ moves from Section~\ref{sec:notations}.
\end{proof} \subsubsubsection{The Stokes formula}

\begin{theorem}[Stokes Formula]\label{Stokes}
For all $\mathcal{A}\in \mathfrak{A}$ and $\gamma\in \mathfrak{G}_1$,

$$\< d \mathcal{A},\gamma\>=\< \mathcal{A},\partial \gamma\>.$$
\end{theorem}

\begin{remark}
As an immediate corollary, we see that $\operatorname{Ker}d$ is exactly the set of arrow diagram formulas -- \ie the elements of $\mathfrak{A}$ that define knot invariants. A result of this kind was already proved by different means in \citep{MortierPolyakEquations} and \citep{VKTG}. It is easy to see that the map $d$ constructed in these earlier works is isomorphic with our map $d^\Lambda$. Also, let us mention that the important Lemma~$3.2$ from \citep{MortierPolyakEquations} can be restated in the present language as
$$\operatorname{Ker} d^\Delta \subset\operatorname{Ker} d^\Lambda \cap \operatorname{Ker} d^{I\!I}.$$
\end{remark}

\begin{proof}

By bilinearity it is enough to prove the formula for an arrow diagram $A$ and a germ $\gamma$.

First assume that $\gamma$ is a $3$-germ. Then $\< d \mathcal{A},\gamma\>=\< d^\Delta\mathcal{A}+ d^\Lambda\mathcal{A}  ,\gamma\>$.
Choose a map $\sigma$ that adds a sign to each arrow of $A$, and (consistently) to each arrow in each diagram in the sum $d A$ -- for this consistency to make sense, we use the particular representative of $d^\Lambda A$ that was constructed in Definition~\ref{def:coboundary}
before it was pushed modulo the triangle relations.

If $\tilde{\gamma}=( \tilde{G}_0, \tilde{G}_1)$ is a partial subgerm of $\gamma$, then:
\begin{equation}\label{stok1}
\left<  \left( d A\right) ^\sigma,\tilde{\gamma}\right>=\left<  \left( d^\Lambda A\right) ^\sigma,\tilde{\gamma}\right>=\left<   A^\sigma,\tilde{G}_1- \tilde{G}_0\right>.
\end{equation}
Indeed, the first equality is obvious, and as for the second,
\begin{enumerate}
\item If $A^\sigma$ is different from both $\tilde{G}_0$ and $\tilde{G}_1$, then both sides of the equality are zero.
\item If $A^\sigma\in\left\lbrace \tilde{G}_0, \tilde{G}_1\right\rbrace $, then there is exactly one term in the sum $d^\Lambda A$ that coincides with  either $\tilde{\gamma}$ or $-\tilde{\gamma}$, and the sign is $+1$ if $A^\sigma=G_1$, $-1$ otherwise.

\end{enumerate}

Similarly, if $\tilde{\gamma}=( \tilde{G}_0, \tilde{G}_1)$ is a subgerm of $\gamma$, then:
\begin{equation}\label{stok2}
\left<  \left( d A\right) ^\sigma,\tilde{\gamma}\right>=\left<  \left( d^\Delta A\right) ^\sigma,\tilde{\gamma}\right>=\left<   A^\sigma,\tilde{G}_1- \tilde{G}_0\right>.
\end{equation}

Finally, notice that if $\tilde{G}$ is a subdiagram of $G_1$ in which less than two arrows remain from the R$_3$ triple, then $\tilde{G}$ is a subdiagram of $G_0$ as well and the corresponding contribution to $\left<  A,\partial \gamma\right>$ is $0$. By definition of a subgerm, one never erases more than one arrow from the R$_3$ triple, so these diagrams do not contribute on the left-hand-side either.

It follows that the Stokes formula is the sum of the equations \eqref{stok1} and \eqref{stok2} (with both sides multiplied by $\operatorname{sign}(\sigma)$) over all possible choices of $\sigma$ and $\tilde{\gamma}$.\\

Now assume that $\gamma=(G_0, G_1)$ is an $i$-germ with $i=1$ or $2$. We treat only the case $i=1$, the other case is similar. Assume without loss of generality that $G_1$ has more arrows than $G_0$. Then $I(\partial\gamma)$ is the formal sum of all subdiagrams of $G_1$ such that the distinguished arrow has not been removed. Indeed, all other subdiagrams of $G_1$ are in $1$-$1$ correspondence with the subdiagrams of $G_0$, and happen in $I(\partial\gamma)$ with opposite coefficients. 

Again let $\sigma$ be a particular completion of $A$ into a Gauss diagram, and let $\tilde{\gamma}=( \tilde{G}_0, \tilde{G}_1)$ be a subgerm of $\gamma$. We shall explain why
\begin{equation}\label{stok3}
\left<  \left( d A\right) ^\sigma,\tilde{\gamma}\right>= \left<  \left( d^I A\right) ^\sigma,\tilde{\gamma}\right>=\left<   A^\sigma,\tilde{G}_1\right>.
\end{equation}
The first equality is obvious. As for the second, recall that $d^I A$ is the sum of all arrow $1$-germs $(A^0,A^1)$ such that $A^1=A$ and $A^1$ has more arrows than $A^0$. Hence $\left<  \left( d^I A\right) ^\sigma,\tilde{\gamma}\right>$
is equal to $1$ if $A^\sigma= \tilde{G}_1$ and to $0$ otherwise. The second equality follows.

As before, the Stokes formula is the alternate sum of equalities \eqref{stok3} over all possible choices of $\sigma$ and $\tilde{\gamma}$.
\end{proof}

\section{General method for producing $1$-cocycle formulas}
\label{sec:method}
From the viewpoint of singularity theory, the Reidemeister moves are the $1$-codimensional strata in a stratification of the space of knots induced by the choice of a projection $\RR^3\rightarrow\RR^2$. Hence, for a linear combination of arrow germs (which is a $1$-cochain defined as an intersection form with some of those strata) to define a $1$-cocycle, it has to vanish on the meridians of the codimension $2$ strata, \ie the boundary of transversal $2$-discs. The description of those strata,  which are the \enquote{simplest} degeneracies of Reidemeister moves, is known as a \textit{higher order Reidemeister theorem} \citep{Fiedler1parameter, Fiedler1parameterPolynomials}, but it also appears as a set of elementary moves in the study of surfaces embedded in $4$-space, known as \textit{Roseman moves} \cite{CarterKamadaSaito}. 

Fortunately, we will see in Section~\ref{sec:reduction}
that most of these strata may be ignored for our purposes (in fact all but three of them), which gives the associated system of equations a reasonable size.

\subsection{Codimension $2$ strata and the associated equations}\label{sec:strata}

Let us recall the different types of strata, together with the corresponding meridians. In the description, we will say that two Reidemeister pictures are \textit{adjacent} if they share exactly one common crossing. As usual, we use the notations of type \enquote{R-I} for Reidemeister moves of knot diagrams, and \enquote{R$_1$} for R-moves of Gauss diagrams.

\paragraph*{1.} Transverse intersection between two $1$-codimensional strata. It is the situation where two Reidemeister moves are far from each other, so that one can choose in what order to perform them. These strata will be respectively denoted by \includegraphics[scale=0.6]{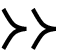}, \includegraphics[scale=0.6]{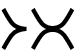}, \includegraphics[scale=0.6]{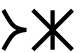}, \includegraphics[scale=0.65]{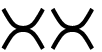}, \includegraphics[scale=0.6]{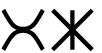}, \includegraphics[scale=0.6]{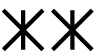}.
For instance, the meridian of a stratum of type \includegraphics[scale=0.6]{R1R2.eps} looks like the following:
\begin{figure}[h!]
\centering 
\psfig{file=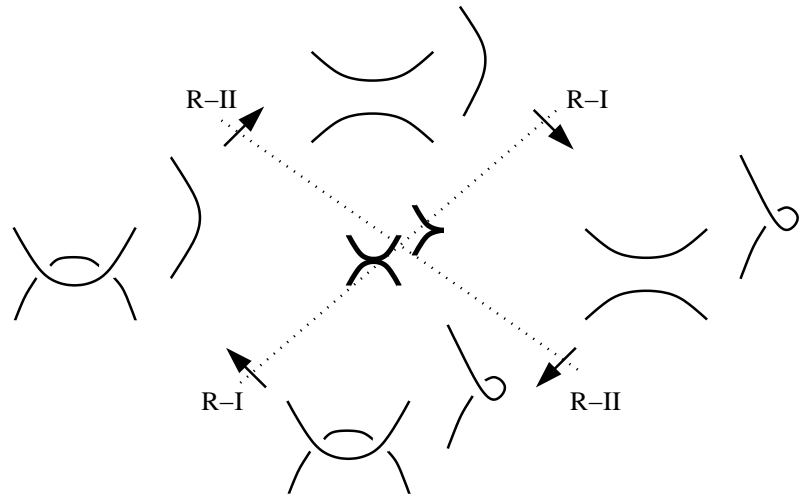,scale=1}
\caption{Meridian of a transverse intersection between $1$-codimensional strata}\label{pic:meridR1R2} 
\end{figure}

\paragraph*{2.} Degenerate cusp -- denoted by \includegraphics[scale=0.6]{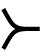}. It is the limit of an R-I and an R-II moves that are adjacent to each other and tend to degenerate at the same time. The meridian is as follows:

\begin{figure}[h!]
\centering 
\psfig{file=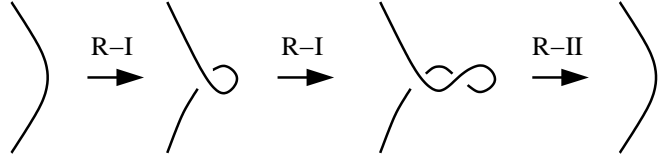,scale=1}
\caption{Meridian of a degenerate cusp}\label{pic:meriddegencusp} 
\end{figure}

This meridian has a particular significance: it is the only one which contains only one R-II move. Therefore it is responsible for the fact that the algebraic number of R-II moves that happen in a loop is not a $1$-cocycle. Roughly speaking, if you count R-II moves and if you want a $1$-cocycle, then you have to count R-I moves as well.
\paragraph*{3.} Cusp with a transverse branch -- denoted by  \includegraphics[scale=0.6]{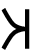}. It happens when an R-III move and an R-I move are adjacent to each other and degenerate simultaneously. Just like the degenerate cusp, this meridian has a unique particularity: it is the only one featuring exactly one R-III move.
 \begin{figure}[h!]
\centering 
\psfig{file=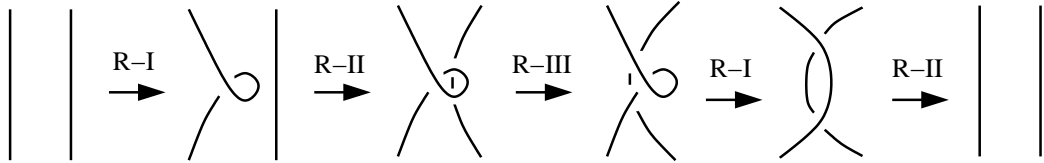,scale=1}
\caption{Meridian of a cusp with a transverse branch}\label{pic:meridtranscusp} 
\end{figure}

\paragraph*{4.} Tangency in an inflection point -- denoted by \includegraphics[scale=0.6]{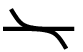}. It is the degeneracy of two R-II moves adjacent to each other.
The meridian has only two Reidemeister moves:

\begin{figure}[h!]
\centering 
\psfig{file=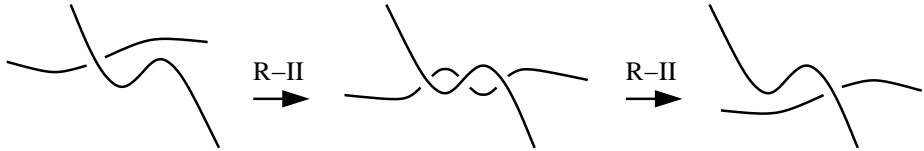,scale=1}
\caption{Meridian of a tangency in an inflection point}\label{pic:meriddegtangency} 
\end{figure}

\paragraph*{5.} Regular tangency with a transverse branch -- denoted by \includegraphics[scale=0.6]{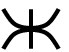}. It is the simultaneous degeneracy of an R-III move and an adjacent R-II move.
\begin{figure}[h!]
\centering 
\psfig{file=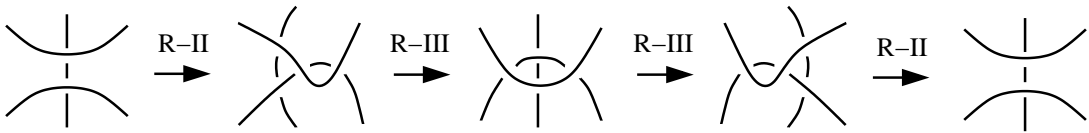,scale=1}
\caption{Meridian of a regular tangency with a transverse branch -- the \textit{cube} equation}\label{pic:meridcube} 
\end{figure}
\begin{remark}
As we will see later (Proposition~\ref{prop:R1R2}), R-I and R-II moves can be safely ignored in our theory. Hence, the heart of this equation is the fact of relating two Reidemeister III moves that share a common edge (as in the middle picture from Fig.\ref{pic:meridcube}).  Note that Fig.\ref{cube} contains simultaneously all the possible local situations. Also, one can see on this picture the polyhedron which is dual to the cube that gave its name to the \includegraphics[scale=0.6]{cubestratum.eps}-equations \citep{Fiedler1parameter, FiedlerQuantum1cocycles}.
\end{remark}
\newpage

\paragraph*{6.} Regular quadruple point -- denoted by \includegraphics[scale=0.6]{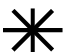}.
It is the simultaneous degeneracy of two adjacent R-III moves.

\begin{figure}[h!]
\centering 
\psfig{file=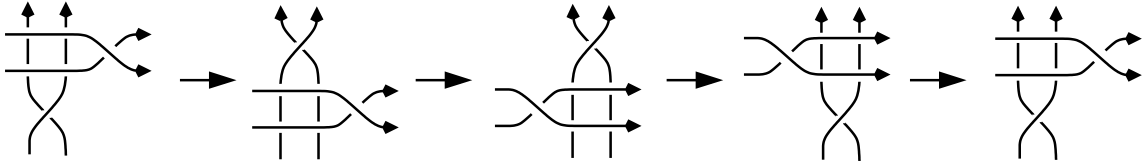,scale=1}
\caption{The meridian of a (positive, braid-like) quadruple point -- the \textit{tetrahedron} equation. Each step is made of two R-III moves.}\label{pic:meridquad} 
\end{figure}

Note that instead of simultaneous degeneracies of Reidemeister pictures, one can think of \includegraphics[scale=0.6]{transcusp.eps}, \includegraphics[scale=0.6]{cubestratum.eps} and \includegraphics[scale=0.6]{quadruple.eps} as one branch sliding over (or under, or between) the branches of a Reidemeister move.

\subsubsection*{How to write down the equations associated with each stratum}
We apply the same process that was used to show that the Polyak algebra is a complete set of equations for Goussarov-Polyak-Viro's invariants of virtual knots \citep{GPV}. No distinction is made between a meridian and the corresponding formal sum of germs in $\mathfrak{G}_1$.

\begin{definition}
\label{def:bystanders}
Let $m\in \mathfrak{G}_1$ be a meridian. The arrows corresponding to unseen crossings in Figs.\ref{pic:meridR1R2} to \ref{pic:meridquad} are called \emph{bystanders}. Roughly speaking, they are arrows which never become distinguished.
\end{definition}

Fix a meridian $m\in \mathfrak{G}_1$. The sum $I(m)$ naturally splits as $$I(m)=\sum  I(m;s)$$ where s runs over all subsets of the set of bystanders, and $I(m;s)$ is the formal sum of subgerms extracted from $m$ where exactly all bystanders that do not lie in $s$ have been removed.

\begin{lemma}\label{lem:IMS}
For any $m$ and $s$ as above, $I(m;s)$ is the image under the map $I$ of a linear combination of meridians.
\end{lemma}
\begin{proof}
One readily checks that
$$I(m;s)=\sum_{s^\prime \subset s} (-1)^{\sharp s\setminus s^\prime}I(
m_{s^\prime}),$$
where $m_{s^\prime}$ is the meridian obtained from $m$ by removing all bystanders that do not lie in $s^\prime$.
\end{proof}

The expressions $I(m;s)$ are the analogue of the relations defining the Polyak algebra in \citep{GPV}. Here it is possible to go a little further.

\begin{definition}We call a \emph{$1$-cocycle equation} any element of $\mathfrak{A}^1$ that is the image under $\TI$ of a linear combination of meridians. 
By definition, an element $\alpha\in \mathfrak{A}^1$ defines a $1$-cocycle if and only if it \emph{satisfies} (\ie is orthogonal to) all $1$-cocycle equations.\end{definition}
By Lemma~\ref{lem:IMS}, the space of equations is exactly spanned by the elements of the form $T(I(m;s))$.

\begin{definition}
The \textit{degree}
of an arrow germ $\alpha=(A^0, A^1)$ is defined by $$\operatorname{deg}
(\alpha)=
\operatorname{max} (\operatorname{deg} (A^0), \operatorname{deg} (A^1)).$$
\end{definition}

\begin{proposition}\label{homogeneous}
The subspace of $1$-cocycle equations admits a family of homogeneous generators (in the sense of the degree defined above). More precisely, any homogeneous part of an equation $T(I (m;s))$ is itself a $1$-cocycle equation.
\end{proposition}

\begin{proof}
For meridians of type \includegraphics[scale=0.6]{R1R1.eps}, \includegraphics[scale=0.6]{R1R2.eps}, \includegraphics[scale=0.65]{R2R2.eps}, \includegraphics[scale=0.6]{degencusp.eps} and \includegraphics[scale=0.6]{degtangency.eps}, the equations $T(I (m;s))$ are themselves already homogeneous.

Choose a meridian $m$ of type \includegraphics[scale=0.6]{R1R3.eps} or \includegraphics[scale=0.6]{R2R3.eps}: the equations $T(I (m;s))$ have terms in only two different consecutive degrees, depending on whether or not an arrow from the R$_3$-triple has been removed. Now let $m^\prime$ be the meridian obtained from $m$ by changing the signs of all arrows involved the R$_3$-move (and do so consistently in all germs of $m$); $m^\prime$ is indeed a meridian, by the criterion for R$_3$-moves given in Section~\ref{sec:notations}.

One sees that the homogeneous part of lower degree in $T(I (m^\prime;s))$ is equal to that of $T(I (m;s))$ (because those parts involve exactly two arrows from the R$_3$ triple), while their parts of higher degree are opposite. Hence 
$$T(I (m;s))=\frac{T(I (m;s))+T(I (m^\prime;s))}{2}+\frac{T(I (m;s))-T(I (m^\prime;s))}{2}$$
is the decomposition of $T(I (m;s))$ in homogeneous parts, which proves the lemma for types \includegraphics[scale=0.6]{R1R3.eps} and \includegraphics[scale=0.6]{R2R3.eps}. 

The proof for the remaining types does not require more ideas.
\end{proof}

\subsection{Reduction of the set of equations} \label{sec:reduction}
Before writing down  the equations associated with all $2$-strata, we are going to show that some equations, as well as some variables, may be ignored.

\subsubsection*{Keeping only one tetrahedron equation}

In the same spirit as small \enquote{sets} of Reidemeister moves can generate all Reidemeister moves (as is shown in \citep{P2}), one naturally expects some $2$-meridians to express as linear combinations of others. Since the equations derived from \includegraphics[scale=0.6]{quadruple.eps} are by far the most numerous and the most complicated, we try to get rid of them in priority. 
In the meridian from Fig.\ref{pic:meridquad}, one could change the orientation and the \enquote{level} (along the projection axis) of each strand. It makes a total of $48$ combinatorially different meridians. Following \cite{FiedlerQuantum1cocycles}, we call the one presented on Fig.\ref{pic:meridquad} \textit{positive braid-like}, because all visible crossings have a positive writhe and because the strand orientations do not forbid that this meridian happens for braids.

The following is proved in \citep{FiedlerQuantum1cocycles}.

\begin{lemma}
The positive braid-like \includegraphics[scale=0.6]{quadruple.eps}-meridian, together with the meridians of other types of $2$-strata, generates all other \includegraphics[scale=0.6]{quadruple.eps}-meridians.
\end{lemma}

While this lemma is inherent to the topology of the space of knots, all the simplifications made hereafter rely on the particular shape of the $1$-cochains we are considering.

\subsubsection*{Getting rid of $1$-germs and $2$-germs}

In the actual computation of the knot invariants given by arrow diagram formulas, it appears that R$_1$- and R$_2$-moves do not take part in the difficulty of the problem: indeed, the kernel of $d^I\oplus d^{I\!I}$ is simply a subspace generated by some single arrow diagrams. Similarly, at the $1$-codimensional level, all the variables related with $R_1$- and $R_2$-moves can be ignored.

\begin{proposition}\label{prop:R1R2}
Any non trivial $1$-cocycle presented by arrow germs admits a presentation using only arrow $3$-germs and partial arrow $3$-germs.
\end{proposition}

\begin{proof}
Let $A$ be an arrow diagram. Denote by $\gamma^{1,2}(A)$ the \textit{set} of all $1$- and $2$-germs involved in the sum $dA$. The collection of all such sets $\gamma^{1,2}(A)$ forms a partition of the set of all $1$- and $2$-germs.

Now let $\alpha\in\mathfrak{A}^1$ be an arrow germ formula, defining a non-trivial $1$-cocycle.\\
\textbf{Claim: }for any two arrow germs $a, b$ lying in the same set $\gamma^{1,2}(A)$, one has
$$\left<\alpha, a \right>=\left<\alpha, b \right>.$$
Indeed, assume that $-a$ and $-b$ are far from each other, \ie kill disjoint sets of arrows in $A$. Then the claim follows from $\alpha$ vanishing on the meridians of the strata \includegraphics[scale=0.6]{R1R1.eps}, \includegraphics[scale=0.6]{R1R2.eps} and \includegraphics[scale=0.6]{R2R2.eps}.
If, on the contrary, $-a$ and $-b$ both kill a common arrow in $A$, then it follows from $\alpha$ vanishing on the meridians of the strata \includegraphics[scale=0.6]{degencusp.eps} and \includegraphics[scale=0.6]{degtangency.eps}.

As a consequence, there is a collection of rational numbers $\left\lbrace \alpha_A \right\rbrace $, almost all equal to zero, indexed by the set of arrow diagrams, such that $$\alpha=\sum \alpha_A \,dA \hspace*{0.5cm}+ \hspace*{0.1cm} \text{some arrow $3$-germs and partial arrow $3$-germs}.$$
It follows now from the Stokes formula (Theorem~\ref{Stokes}), that the combination $\alpha-\sum \alpha_A \,dA$ still defines a $1$-cocycle, which is cohomologous to that defined by $\alpha$.
\end{proof}

\subsubsection*{How to ignore some R$_3$ strata as well}
Thanks to Proposition~\ref{prop:R1R2}, we may forget about the R$_1$- and R$_2$-moves involved in all meridians. It makes some equations particularly simple, and enables us to kill yet another family of germs.

\begin{conventiondef}From now on, we will always consider arrow germ formulas that are only made of $3$-germs and partial germs, called \emph{arrow 3-germ formulas}. Also, to remove the indeterminacy due to the triangle relations, we consider only those partial arrow germs that switch an arrowhead and an arrowtail, called \emph{monotonic partial (arrow) germs}. It is easy to see that they form a basis of $\mathfrak{A}^\Lambda    /\nab$.
\end{conventiondef}

\begin{proposition}\label{prop:3}
A germ cannot participate in an arrow $3$-germ formula if it satisfies any of the following conditions. 

\begin{enumerate}
\item A non-distinguished arrow is isolated (\ie in a position to be killed by R$_1$).
\item Two non-distinguished arrows are in a position to be killed by R$_2$. 
\item (Specific to partial germs) -- a distinguished arrow is isolated on one side of the germ.
\item (Specific to $3$-germs) -- a distinguished arrow is isolated on one side of the germ and, if one removes it, the two remaining distinguished arrows are in a position to be killed by R$_2$. 
\end{enumerate}

\end{proposition}

\begin{proof}
The claims are direct consequences of the fact that a $1$-cocycle vanishes on the meridians of \includegraphics[scale=0.6]{R1R3.eps} (point $1$),  \includegraphics[scale=0.6]{R2R3.eps} (point $2$), and \includegraphics[scale=0.6]{transcusp.eps} (points $3$ and $4$). Indeed, once $1$-germs, $2$-germs and non-monotonic germs have been removed, the homogeneous equations extracted from these meridians (Proposition~
\ref{homogeneous}) contain only one term left, and these one-term equations kill all diagrams of the indicated forms. The only non-trivial observation, for point $3$, is that in a \includegraphics[scale=0.6]{transcusp.eps}-meridian, the arrow born and killed by R$_1$ is always one of the two arrows that are distinguished in the monotonic partial subgerms of the R$_3$-move.
\end{proof}

Altogether, we have proved the following theorem.

\begin{theorem}\label{thm:system}
The linear system with
\begin{itemize}

\item as variables: arrow $3$-germs and monotonic partial arrow $3$-germs that do not satisfy any of the four conditions from Proposition~\ref{prop:3}, and
\item as equations: all homogeneous equations extracted from the meridians of \includegraphics[scale=0.6]{R3R3.eps}, \includegraphics[scale=0.6]{cubestratum.eps} and
\includegraphics[scale=0.6]{quadruple.eps} by Proposition~\ref{homogeneous}, 

\end{itemize}
is a complete system for the $1$-cohomology classes of knots given by arrow germ formulas.
\end{theorem}

\subsection{Some remarks and questions}

\subsubsection*{$1$-cocycles of virtual knots}
Because it is based on a Gauss diagram construction, our method produces by nature $1$-cocycles in the space of \textit{virtual} knots -- and hence, of course, of actual knots as well. Indeed, they are solutions to all Gauss diagrammatic equations associated with the meridians of $2$-strata in the space of knots, regardless of whether the  diagrams in those meridians actually correspond to knot projections. 

\newpage

In fact, arrow germ formulas for $1$-cocycles of actual knots should only satisfy a subset of our equations, but it is very hard to determine that subset. An analog of this phenomenon can be observed for arrow diagram formulas: Polyak-Viro's formula for $v_3$ \citep{PolyakViro} is not an invariant of virtual knots, and as a result, the proof that it is actually an invariant for actual knots requires some special knowledge of arrow diagram identities -- see \citep{Ostlund}.

%
%
%
%

\subsubsection*{ Arrow germ presentations of the trivial $1$-cocycle}

By the Stokes formula (Theorem~\ref{Stokes}), all \enquote{derivatives} of linear combinations of arrow diagrams are trivial $1$-cocycles.

\begin{question}
Is the converse true? Namely, is it true that any arrow germ formula defining a trivial $1$-cocycle is actually the derivative of a linear combination of arrow diagrams? 
\end{question}

Although this would seem quite natural, it is not obvious -- and in fact we have no clue whether it is true or not. 

However, if it is not, then it would still be interesting to understand what are actually the \enquote{primitives} of those trivial cocycles. Can some combinations of them lead to new knot invariants?

\subsubsection*{Finite-type cohomology classes}

Recall that the degree of an arrow germ $(A^0, A^1)$ is defined by $\operatorname{max} (\operatorname{deg} (A^0), \operatorname{deg} (A^1))$. We end this section with the most important question of this article.

\begin{conjecture}\label{conjecture}
Every $1$-cocycle presented by an arrow germ formula of degree $n$ is of finite-type, of degree at most $n$. 
\end{conjecture}

Section~\ref{sec:TT} will provide a number of clues in favor of this conjecture. In general, here is already one reason to believe that arrow germs and Vassiliev's theory should be related: although they are most of the time thought of via planar projections, one should not forget that \textit{local writhes} are actually well-defined signs associated with the two desingularizations of a double point, for singular knots in any oriented $3$-manifold (see \citep{FiedlerSSS}, Lemma~$1$). In other words, (products of) writhe signs are natural co-orientations of Vassiliev's singular strata: it makes it possible to understand what is the oriented boundary of an arrow germ formula in the first stage of Vassiliev's filtration (see Section~\ref{sec:FT}).

\section{Teiblum-Turchin's cocycle}\label{sec:TT}

The first -- and, so far, the only -- example of a finite-type integral $1$-cocycle of knots is the so-called Teiblum-Turchin cocycle $v_3^1$. No formula is known to actually evaluate it, except mod $2$ \citep{Vassiliev, Turchin}. In this section, we compute completely the linear system described in Theorem~\ref{thm:system} in homogeneous degree three. Then we study one of the solutions, called $\alpha_3^1$, which we conjecture is equal to the Teiblum-Turchin cocycle. Two properties are proved to support this conjecture: first, the evaluation of $\alpha_3^1$ on the rotation of a long knot $K$ around its axis gives the Casson invariant $v_2(K)$, which is conjectured to hold for $v_3^1$ in \cite{Turchin}; second, we show that $\alpha_3^1 \operatorname{mod} 2$ is of finite type, which implies the equality $\alpha_3^1=v_3^1$ over $\ZZ_2$.

\subsection{Arrow germ formulas of degree three}

Note that the lowest degree of a germ involved in the meridian of a \includegraphics[scale=0.6]{R3R3.eps}-stratum is $4$, hence we are left with the study of the two strata  \includegraphics[scale=0.6]{cubestratum.eps} and
\includegraphics[scale=0.6]{quadruple.eps} -- \ie the cube and tetrahedron equations.

\subsubsection*{The cube equations}

\begin{figure}[h!]
\centering 
\psfig{file=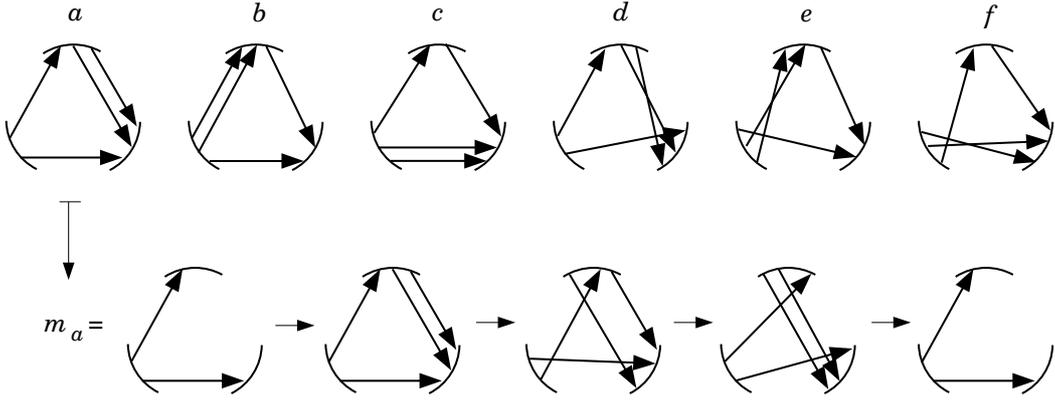,scale=0.8}
\caption{Six diagrams representing the essentially different cube meridians, with one example}\label{cubeqs} 
\end{figure}

There are $6\times 2^3$ meridians of type \includegraphics[scale=0.6]{cubestratum.eps} from the viewpoint of Gauss diagrams. The factor $6$ corresponds to the six diagrams on Fig.\ref{cubeqs} (the lower row of the picture shows as an example the meridian associated with the first diagram), and each factor $2$ is a binary choice, that is actually not essential:
\begin{enumerate}
\item For each picture, there are exactly two choices of signs that make the R$_2$- and R$_3$-moves possible, and this choice has no influence in the resulting equation.
\item For each picture, say \includegraphics[scale=0.6]{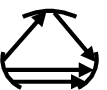}, one can switch the positions of the two arrow ends that are alone -- in the example, \includegraphics[scale=0.6]{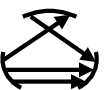}. Again the associated equations are identical.
\item Reversing all the arrows in the pictures yields the same equations, with all arrows reversed. In particular, note that the set of variables that we have not banned from the system (Theorem~\ref{thm:system}) is stable under the arrow-reversing operation.
\end{enumerate}
The equations of degree $3$ derived from the six meridians on Fig.\ref{cubeqs} are respectively given in Fig.\ref{cubeqs2} -- all of them are homogeneous parts of degree $3$ of  $T(I (m;s))$ with $s=\emptyset$ (Definition~\ref{def:bystanders}).
Note that substitutions can be made so that equations $d)$ and $e)$ contain both only (four) partial germs. Also, note that $c)$ contains three different equations, indexed by the choice of a point at infinity.
\begin{figure}[h!]
\centering 
\psfig{file=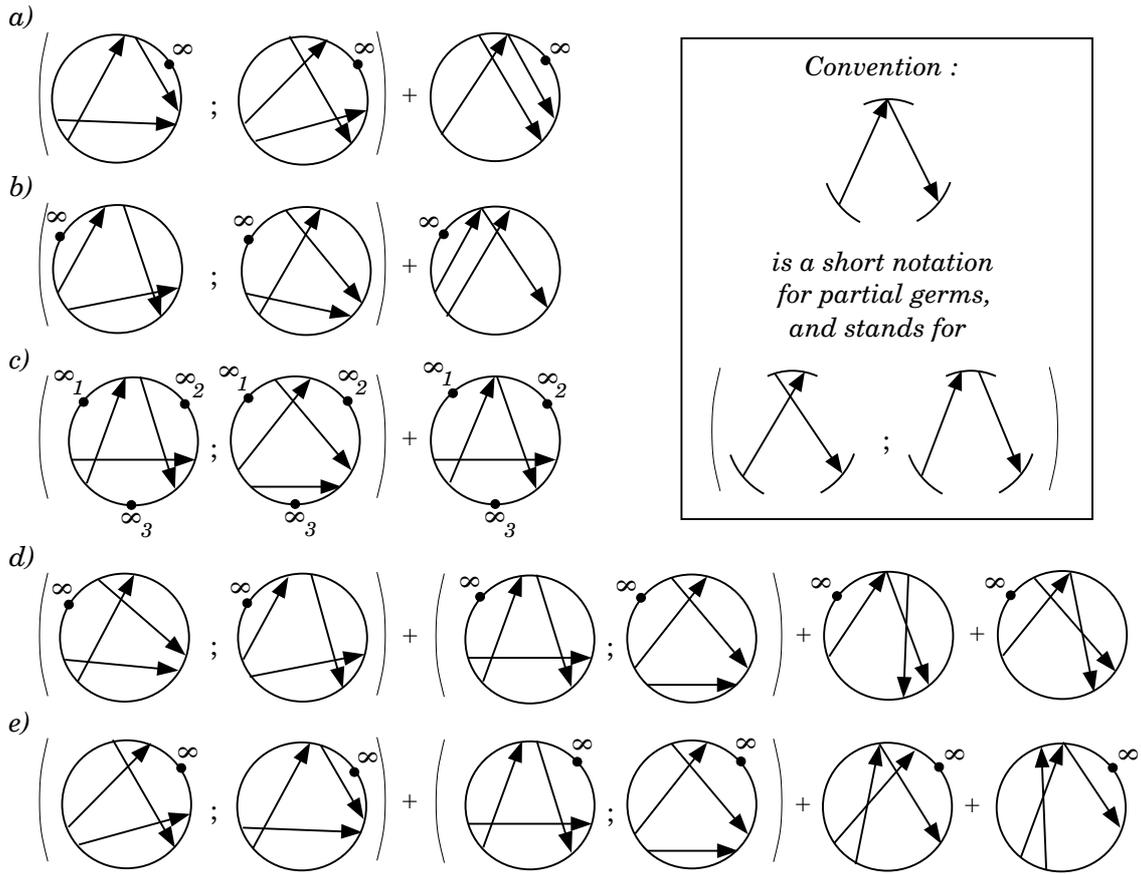,scale=0.76}
\caption{Half of the degree $3$ cube equations -- for the other half, reverse all the arrows}\label{cubeqs2} 
\end{figure}

\begin{remark}
Our notational convention for partial arrow germs (Fig.\ref{cubeqs2}) is not anecdotic. It corresponds to the co-orientation of R$_3$-strata given by the sign $\mathrm{w}(e)\varepsilon(e)$ (see Remark~\ref{rem:coor}).
Arrow germ formulas of degree $n$ simply count the intersection index (with R$_3$ strata endowed with this coorientation), weighted by the sign of $n-2$ arrows.
\end{remark}

\subsubsection*{The tetrahedron equations}

A Gauss diagrammatic description of all meridians coming from the positive braid-like quadruple point has been written in \citep{FiedlerQuantum1cocycles}
\textit{(pp. 35--46)}. One readily sees that all degree $3$ equations coming from those meridians are simple linear combinations of the cube equations (Fig.\ref{cubeqs2}), except for the one corresponding to global type II, with point at infinity $2$, and global type IV, with point at infinity $3$. (\citep{FiedlerQuantum1cocycles},
\textit{pp. 37--38 and 41--42}), which are obtained from each other by reversing the orientation of all the arrows. They are shown in Fig.\ref{tetraeqs}, using the same notational convention as in Fig.\ref{cubeqs2}.

\begin{figure}[h!]
\centering 
\psfig{file=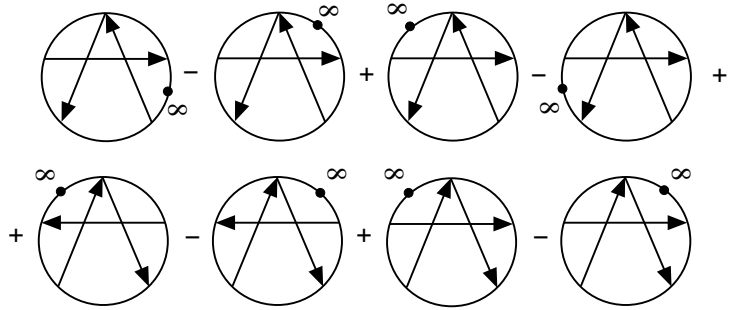,scale=0.76}
\caption{One degree $3$ tetrahedron equation -- for the other one, reverse all the arrows}\label{tetraeqs} 
\end{figure}

It is proved in \citep{Vassiliev} that the group of finite-type integral $1$-cocycles of degree $3$ is isomorphic to $\ZZ$. Hence, if Conjecture~\ref{conjecture} is true, one cannot expect more than essentially one non-trivial solution to the system of equations on Figs.\ref{cubeqs2} and \ref{tetraeqs}. Therefore, we shall only extract and study one such solution.

\begin{figure}[h!]
\centering 
\psfig{file=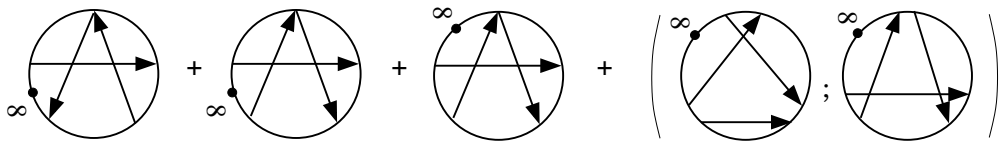,scale=0.76}
\caption{The first non-trivial arrow germ formula $\alpha_3^1$}\label{pic:a31} 
\end{figure}

\begin{theorem}
The arrow germ formula $\alpha_3^1$ on Fig.\ref{pic:a31} defines a $1$-cocycle of knots.
\end{theorem}

\begin{proof}
One readily checks that it satisfies all equations from Figs.\ref{cubeqs2} and \ref{tetraeqs} as well as the ones obtained from them by reversing the orientations of the arrows. Since the arrow germs involved in $\alpha_3^1$ do not take part in any other meridian than those of types \includegraphics[scale=0.6]{cubestratum.eps} and
\includegraphics[scale=0.6]{quadruple.eps}, the theorem follows.
\end{proof}

\newpage

\subsection{
$\alpha_3^1(\operatorname{rot}(K))=v_2(K)$}

For any long knot $K$, there is a loop $\operatorname{rot}(K)$ canonically defined by rotating $K$ positively around its axis. We repeat in Fig.\ref{pic:rotK1} the diagrammatic description of $\operatorname{rot}(K)$ given in \citep{FiedlerQuantum1cocycles}.
To convince oneself that this loop amounts to one rotation around the axis, it can be helpful to think of the two isotopic bands (or framed long knots) at the top-left of Fig.\ref{pic:rotK1}.
\begin{figure}[h!]
\centering 
\psfig{file=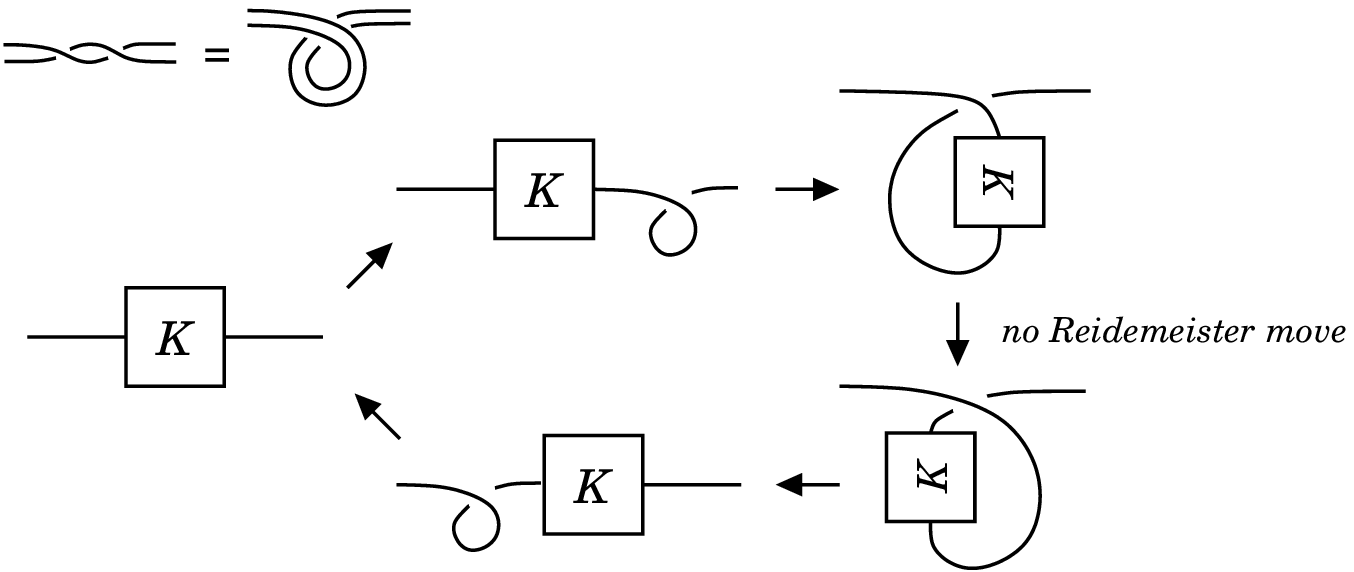,scale=1}
\caption{The loop $\operatorname{rot}(K)$}\label{pic:rotK1} 
\end{figure}

On the other hand, $v_2(K)$ denotes the Casson invariant of $K$, which is the only non-trivial Vassiliev invariant of degree $2$; it is given by the following arrow diagram formula (\citep{PolyakViro}, \textit{Theorem~1 and remark on p.451}):
$$v_2(K)=\left<\raisebox{-1ex}{\includegraphics[scale=1]{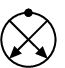}}\,, K\right>.$$

\begin{theorem}
For all knots $K$, $$\alpha_3^1(\operatorname{rot}(K))=-v_2(K).$$
\end{theorem}

In \citep{Turchin}, V.Turchin conjectures that the Teiblum-Turchin cocycle $v_3^1$ satisfies the above equation (with the opposite sign). His conjecture would follow from our Conjecture~\ref{conjecture} and this theorem.

\begin{proof}
In all this proof the dot on each diagram stands for the point at infinity.
R$_3$-moves only happen in the two horizontal arrows on Fig.\ref{pic:rotK1}. For the top arrow, it is easy to see that all of them have two arrows from the R$_3$-triple in position \raisebox{-1.2ex} {\includegraphics[scale=1]{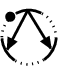}}, hence they cannot be seen by $\alpha_3^1$.

In the bottom horizontal arrow on Fig.\ref{pic:rotK1}, there is exactly one R$_3$-move for each crossing in the initial diagram of $K$. Moreover, the R$_3$-triple is in position 
\begin{enumerate}
\item \raisebox{-1ex}{\includegraphics[scale=1.2]{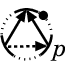}} for each arrow $p$ in position  \raisebox{-1ex}{\includegraphics[scale=1.2] {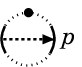}} in the initial diagram, and

\item \raisebox{-1ex}{\includegraphics[scale=1.2]{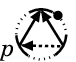}} for each arrow $p$ in position \raisebox{-1ex}{\includegraphics[scale=1.2]{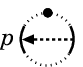}}.
\end{enumerate}
 
The formula $\alpha_3^1$ is oblivious of R$_3$-moves in situation $2$. As for situation $1$, the last three germs of Fig.\ref{pic:a31} do not participate, and the first germ counts the number of arrows $q$ in position \raisebox{-1.3ex}{\includegraphics[scale=1.2]{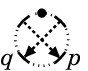}} with respect to $p$ and the point at infinity. Obviously arrows $q$ in such a position can only come from the initial diagram (all other arrows, that are due to the arc slide over $K$, have their head \enquote{immediately after} the point at infinity). Hence the theorem is proved modulo $2$. We have to understand the signs.

For each couple $p$, $q$ as above, corresponding to the situation \raisebox{-2ex}{\includegraphics[scale=1.2]{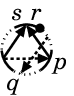}},  $\alpha_3^1$ counts the sign $$w(p)\times w(q)\times w(r)\times \varepsilon,$$
where $w(\cdot)$ is the sign of an arrow, and $\varepsilon$ is equal to $+1$ if the arrows $p$ and $r$ cross each other before the move, $-1$ otherwise (by the convention on Fig.\ref{cubeqs2}).
Since $v_2(K)$ counts $w(p)\times w(q)$ for all such couples, the claim is that one always has $w(r)\times \varepsilon= -1$.

By a diagram isotopy, one can assume that from the beginning, every crossing in the diagram of $K$ has its two branches oriented \enquote{to the right} as in Fig.\ref{pic:rotK9} (this only serves to reduce the number of situations to consider). Then there are only two possible situations, depending on $w(p)$; they are depicted on Fig.\ref{pic:rotK9}, which shows the diagrams right \textit{before} the R$_3$-move is performed. In each of these situations, one has $w(r)=+1$ and $\varepsilon=-1$, hence $w(r)\times \varepsilon= -1$ as announced.
\begin{figure}[h!]
\centering 
\psfig{file=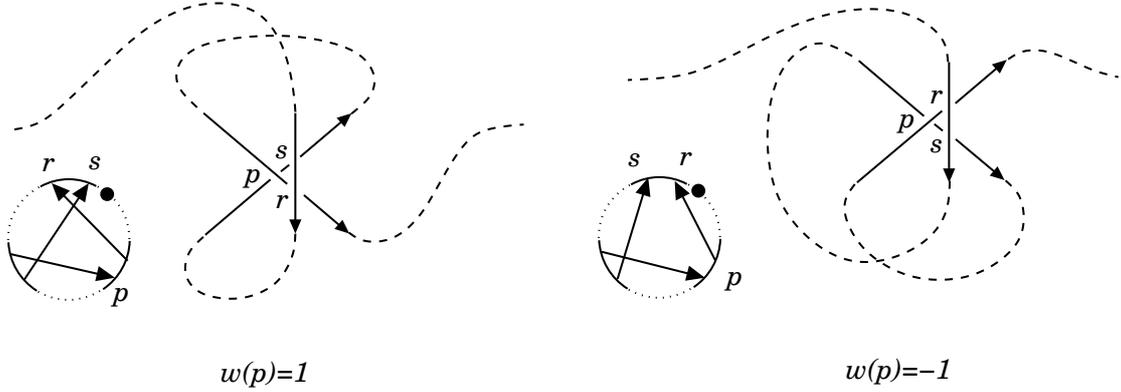,scale=0.8}
\caption{Two possible situations if the branches are directed to the right near every crossing }\label{pic:rotK9} 
\end{figure}

\end{proof}

\subsection{The reduction of 
$\a$ mod $2$ is a finite-type $1$-cocycle}\label{sec:FT}

This section is devoted to prove that $\alpha_3^1$ is a potential formula for the Teiblum-Turchin cocycle $v_3^1$ over the integers.
\begin{theorem}\label{thm:mod2}
The reduction of $\alpha_3^1$ mod $2$ is a finite-type $1$-cocycle of degree $3$.
\end{theorem}

\begin{corollary}
$\alpha_3^1$ and $v_3^1$ coincide over $\ZZ_2.$
\end{corollary}

\begin{proof}[Proof of the corollary (the notations are introduced thereafter, after \citep{Vassiliev1990})]
It appears in the proof of Theorem~\ref{thm:mod2} that the reductions mod $2$ of $\alpha_3^1$ and $v_3^1$ have the same principal part in $\sigma_3\setminus \sigma_2$ (Fig.\ref{pic:ft3}; see \citep{Vassiliev}, \textit{Stabilization formula over $\ZZ_2$ and Proposition~6}). Since the reduced homology group $\tilde{H}_ {\omega-3} (\sigma_2; \ZZ_2)$ is trivial \citep{Vassiliev1990}, it follows that $\alpha_3^1$ and $v_3^1$ are cohomologous over $\ZZ_2$.
\end{proof} 

\begin{remark}
This gives an answer to a question that was asked by T.Fiedler ten years ago \citep{privateFiedler}: yes, there is a description of $v_3^1 \operatorname{mod} 2$ that does not involve differential geometric aspects. However, it will appear in the proof of the theorem that the homological puzzle inside Vassiliev's resolved discriminant still contains some differential geometric pieces. 
\end{remark}

The rest of this section is devoted to the proof of Theorem~\ref{thm:mod2}.

Let us recall how Vassiliev's finite-type cohomology classes are defined -- after \citep{Vassiliev1990} and \citep{VassilievBook}. The space of all smooth immersions $\RR\rightarrow \RR^3$ is denoted by $\mathcal{K}$. To think of it as a manifold, Vassiliev constructs approximations of this space by finite-dimensional affine spaces of polynomial immersions, and shows that part of his construction stabilizes as the dimension tends to infinity. We will denote here this finite dimension of $\mathcal{K}$ by the letter $\omega$. The space $\mathcal{K}$ has a natural stratification induced by the degree of non genericity of the immersions. Non generic immersions form the \textit{discriminant} $\Sigma$ of $\mathcal{K}$.

The crucial observation is made that using Alexander duality, an $n$-cocycle in the space of knots -- that is, in $\mathcal{K}\setminus \Sigma$ -- is the linking form with some $(\omega - n - 1)$-cycle in the one-point compactification of $\Sigma$, relative to the compactification point.
From then on, the idea of Vassiliev is to measure the complexity of those cycles: 
\begin{enumerate}
\item Each elementary part of $\Sigma$ is artificially \textit{dilated}, \ie replaced with a cell of a higher dimension, which depends essentially on its complexity in the stratification.

\item This new space made of dilated cells -- called a \textit{resolution} of the discriminant, is endowed with a filtration, also related with the initial complexity of each cell in the stratification.

\item An $n$-cocycle in the space of knots is said to be of finite type no greater than $k$ if the associated $(\omega - n - 1)$-cycle in $\Sigma$ has a lift in the $k$-th stage of the filtration of the resolution.
\end{enumerate}

How to relate this with the formula defining $\a$?

By nature, arrow germ formulas are intersection forms with certain relative $1$-codimensional chains in the space $(\mathcal{K}, \Sigma)$. The boundary of these chains in $\Sigma$, of course, does not have itself a boundary in $\Sigma$. But it does have a non trivial boundary in the resolution of $\Sigma$. The question is, can this \enquote{artificial}  boundary be pushed deeper into the filtration so as to vanish (see Fig.\ref{pic:FT})?

\begin{figure}[h!]
\centering 
\psfig{file=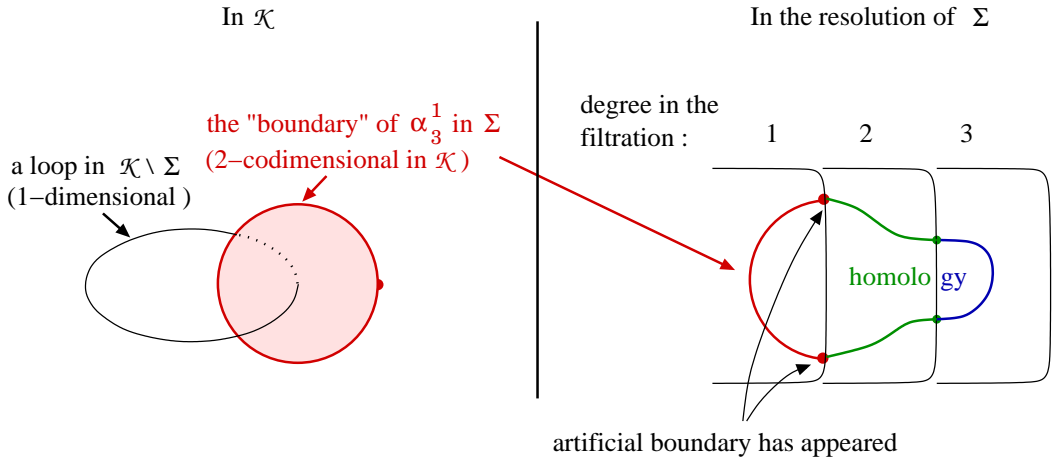,scale=1}
\caption{Finite-type $1$-cocycles of degree $3$}\label{pic:FT} 
\end{figure}

Note that \citep{Vassiliev} adopts the opposite approach: D.Teiblum and V.Turchin had found a candidate for the \enquote{principal part} of a cycle, \ie a relative homology class in $\sigma_3\setminus \sigma_2$, and V.A.Vassiliev computed its successive boundaries in the associated spectral sequence until he found an actual formula for the $1$-codimensional chain on the left of Fig.\ref{pic:FT}.
\subsubsection*{Vassiliev calculus}

In practice, each cell (\textit{J-block} in \citep{VassilievBook}) in the resolution of $\Sigma$ is graphically represented by a based circle marked with finitely many distinguished points; a point can be distinguished in two ways that are not mutually exclusive: being marked with a \textit{star}, or being linked by a \textit{chord} to one or more other distinguished points. Roughly speaking, the chords and stars indicate where the singular knots that \textit{respect} the diagram should have double points and vanishing derivative.

Two essential quantities are associated with such a cell: its \textit{complexity} which is its degree in the filtration, and its dimension. Denote by $\abs{A}$ the number of points in the circle that bound at least one chord; by $\sharp A$, the number of connected components of the abstract graph of chords; by $b$, the number of stars.
\begin{itemize}
\item The complexity $i$ of a cell is given by the formula:
$$i = \abs{A} - \sharp A +b.$$
\item Its dimension is given by:
$$\operatorname{dim} = (\omega- 3i) + (\text{number of distinguished points}) + (\text{number of chords }+\text{ number of stars }-1).$$

\end{itemize}
The first summand $(\omega- 3i)$ is the dimension of the set of knots in $\mathcal{K}$ that \textit{respect} the chord diagram. It explains the choice of $i$ as a complexity. The first two summands together, $(\omega- 3i)+(\text{number of distinguished points})$, give the dimension of the set of knots in $\mathcal{K}$ that respect the chord diagram \textit{up to equivalence}, which is defined as usual by positive homeomorphisms of the based circle.
Finally, the last summand corresponds to an \textit{artificial simplex} whose vertices are the chords and the stars of the diagram. Roughly speaking, it is this simplex that enables high dimensional chains to \enquote{live} in the highly degenerate parts of $\Sigma$.

\begin{remark}
Despite the presence of $\omega-\ldots$ in the dimension formula, one must not forget that those cells are not embedded in $\mathcal{K}$, hence they do not have a codimension there. With one exception though: the cells of complexity $1$ are  naturally homeomorphic with actual parts of $\Sigma$ (of codimension $1$ in $\mathcal{K}$), since the artificial simplex has dimension $0$ in this case.
\end{remark}

\begin{figure}[h!]
\centering 
\psfig{file=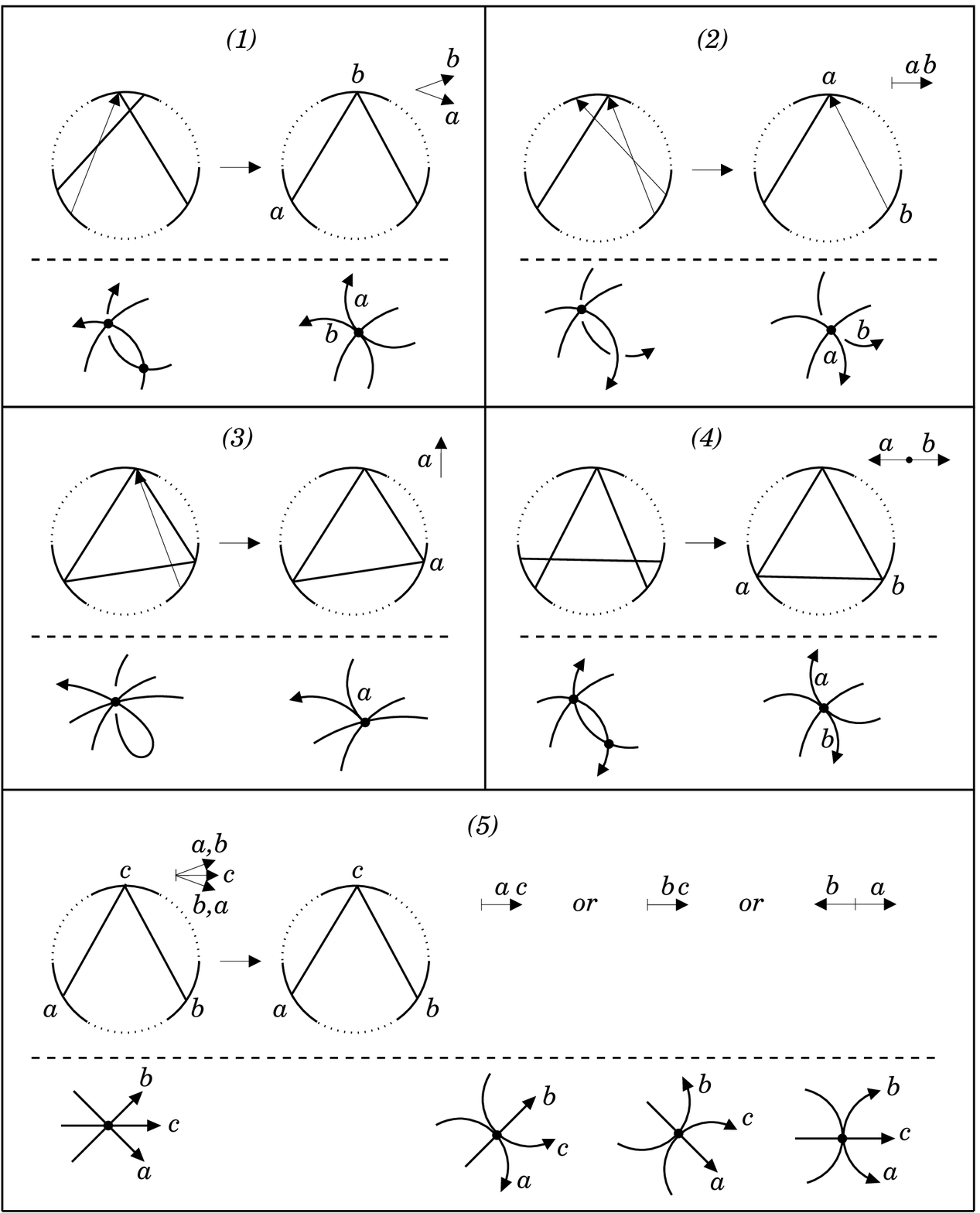,
scale=0.82}
\caption{Exceptional degeneracies accompanied with differential geometric conditions}\label{pic:geom} 
\end{figure}

The crucial point to make actual computations is that the boundary of such a cell is not in general a combination of other similar cells. Rather, it is a combination of \textit{subvarieties} of cells, sometimes defined by conditions of (differential) geometric nature. 
This phenomenon requires additional notations on cell diagrams. In general, to compute the dimension of a cell diagram enhanced with geometric conditions, one first computes the dimension of the ambient cell, and then the codimension induced by the additional conditions.

There are two basic additional decorations on a cell diagram:
\textit{oriented chords}, and \textit{dashed chords}. They indicate that, inside the ambient cell, we only consider the knots with a classical crossing --  respectively, double point -- as indicated by the chord. If the endpoints of this chord do not coincide with any other particular point of the cell diagram, then this additional condition is of codimension $0$ --  respectively, $1$.

First let us recall the non-geometric types of degeneracies.
\begin{enumerate}
\item A chord turns into a dashed chord. It corresponds to the loss of one dimension in the artificial simplex -- the knots, however, still have the corresponding double point. Note that this may or may not involve a $-1$ jump in the filtration degree.
\item An oriented chord turns into a dashed chord. It means that the branches of the corresponding crossing tend to coincide and form a double point.
\item An edge in the chord diagram is shrunk to a point -- \ie two distinguished points tend to coincide. The special rules that accompany this degeneracy is described in \citep{Vassiliev} (\textit{pp. 44-45}). 

\end{enumerate}

We now give in Fig.\ref{pic:geom} a list of examples of exceptional situations that involve differential geometry. It may not be exhaustive in general, but it covers every typical case in our computations related to $\a$.
The superscripts over the chord diagrams in Fig.\ref{pic:geom} respect the notations from \citep{Vassiliev} and should be read respectively as follows ($f$ denotes in each case a parametrization of the knot).

\begin{enumerate}
\item $f^\prime(a)$ and $f^\prime(b)$ have colinear projections to $\RR^2$ with non-negative coefficients and, moreover, the partial derivative at $b$ with respect to the third coordinate is greater than that at $a$.
\item $f^\prime(a)$ and $f^\prime(b)$ have colinear projections to $\RR^2$ with non-negative coefficients.
\item $f^\prime(a)$ is directed \enquote{up} (in the direction of projection, towards the eye).
\item $f^\prime(a)$ and $f^\prime(b)$ have opposite directions.
\item Diagram on the left: the projection to $\RR^2$ of $f^\prime(c)$ is a linear combination of the projections of $f^\prime(a)$ and $f^\prime(b)$, with non-negative coefficients.

Diagrams on the right: see point $2$. 

\end{enumerate}

With those indications, one can now readily check that the following pieces (Figs.\ref{pic:ft3} to \ref{pic:ft1}) fit well into a cycle in the resolved discriminant, and that the part in $C_{\omega-2}(\sigma_1)$ (Fig.\ref{pic:ft1}) is indeed the boundary of the chain presentation of $\alpha_3^1$ (Fig.\ref{pic:a31}). The only non-trivial point that remains to be made clear is why the last diagram on Fig.\ref{pic:ft2} does not have a piece of boundary in the cell \raisebox{-1.2ex}{\includegraphics
[scale=1]{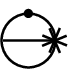}}. The reason is that because of the superscript \raisebox{-1.5ex}{\includegraphics
[scale=1]{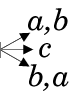}}, this degeneracy would involve too large a jump in dimension: the cusp at the star point should be tangent to the other branch that passes through this point, a condition of codimension $1$ in a cell of dimension $\omega-3$.

\begin{figure}[h!]
\centering 
\psfig{file=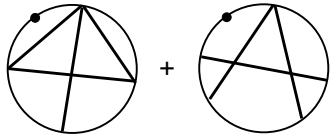,
scale=0.76}
\caption{The Alexander dual cycle to $\alpha_3^1$ -- \textit{principal} part in $C_{\omega-2}(\sigma_3\setminus \sigma_2)$}
\label{pic:ft3} 
\end{figure}

\begin{figure}[h!]
\centering 
\psfig{file=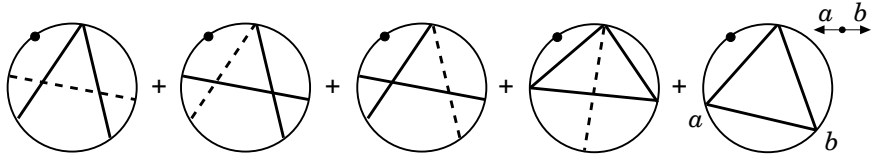,
scale=0.76}
\caption{The common boundary of Figs.\ref{pic:ft3} and \ref{pic:ft2} in $C_{\omega-3}(\sigma_2)$ (see \citep{Vassiliev}, \textit{Proposition~7})}
\label{pic:ft25} 
\end{figure}

\begin{figure}[h!]
\centering 
\psfig{file=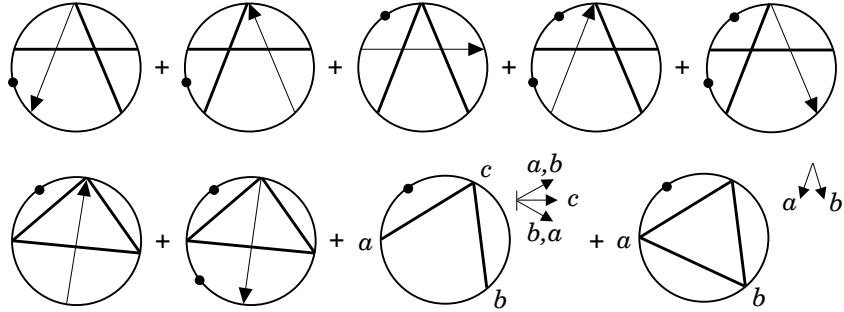,
scale=0.76}
\caption{The Alexander dual cycle to $\alpha_3^1$ -- part in $C_{\omega-2}(\sigma_2\setminus \sigma_1)$. Two points at infinity on the same diagram means the formal sum of two choices. The superscript of the last diagram means that the direction \enquote{down} (of projection to $\RR^2$) is a non-negative combination of $f^\prime (a)$ and $f^\prime (b)$}
\label{pic:ft2} 
\end{figure}

\begin{figure}[h!]
\centering 
\psfig{file=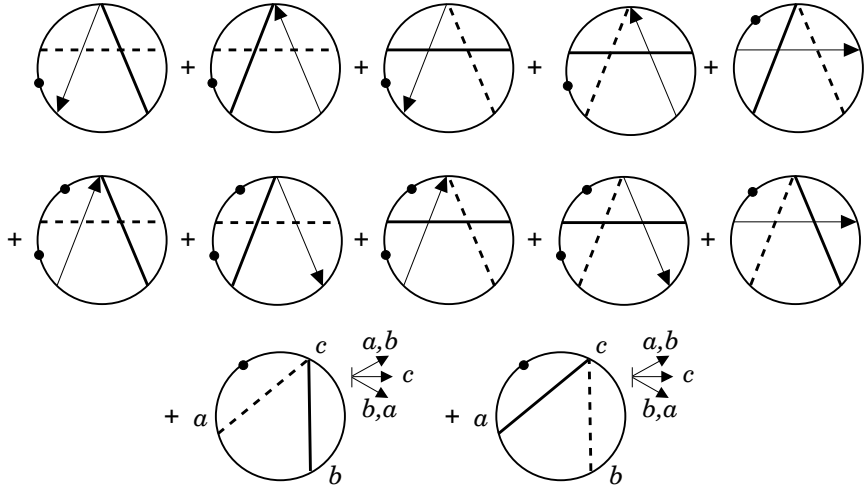,
scale=0.76}
\caption{The common boundary of Figs.\ref{pic:ft2} and \ref{pic:ft1} in $C_{\omega-3}(\sigma_1)$}
\label{pic:ft15} 
\end{figure}

\begin{figure}[h!]
\centering 
\psfig{file=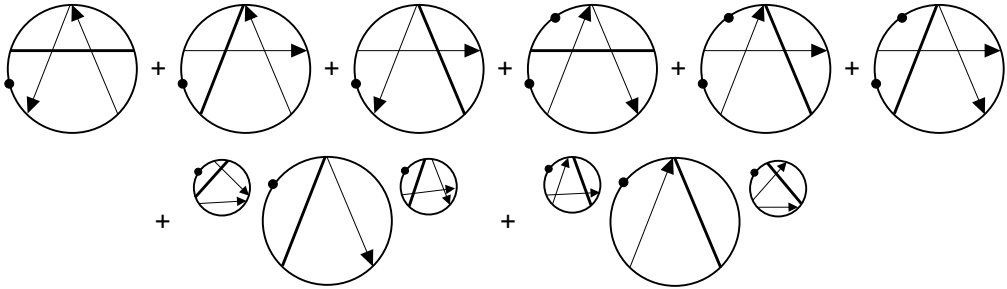,
scale=0.76}
\caption{The Alexander dual cycle to $\alpha_3^1$ -- part in $C_{\omega-2}(\sigma_1)$. The superscripts here mean that we consider only singular knots with this prescribed behaviour on each side of the stratum}
\label{pic:ft1} 
\end{figure}

\bibliographystyle{plain}
\bibliography{bibli}

\end{document}